\documentclass{amsart}

\usepackage{graphicx,amssymb,upref}

\usepackage{mathpazo} %Fonts
\usepackage{amsmath,amssymb,amsfonts,amsthm,amscd}
\usepackage{mathrsfs,latexsym,url,setspace,xcolor,graphicx}
\newtheorem{theorem}{Theorem}
\newtheorem{theoremA}{Theorem}
\newtheorem*{theorem*}{Theorem*}
\newtheorem*{proposition*}{Proposition}

\newtheorem{lemma}{Lemma}
\newtheorem{corollary}{Corollary}
\theoremstyle{remark}

\newtheorem{remark}{Remark}
\newtheorem{definition}{Definition}

\newcommand{\ackname}{Acknowledgements}

\newcommand{\D}{\Omega}

\title[Compactness of Hankel Operators with Conjugate Holomorphic Symbols]{Compactness of Hankel Operators with Conjugate Holomorphic Symbols on Complete Reinhardt Domains in $\mathbb{C}^2$}
\author{Timothy G. Clos}

\email{timothy.clos@utoledo.edu}  
\date{\today}

\begin{document}
\maketitle	
\begin{abstract}
	  In this paper we characterize compact Hankel operators with conjugate holomorphic symbols
	on the Bergman space of bounded convex Reinhardt domains in $\mathbb{C}^2$.  We also characterize compactness of Hankel operators with conjugate holomorphic symbols on smooth bounded pseudoconvex complete Reinhardt domains in $\mathbb{C}^2$.
\end{abstract}

\section{Introduction}	
	
We assume $\Omega\subset \mathbb{C}^2$ is a bounded convex Reinhardt domain.  We denote the Bergman space with the standard Lebesgue measure on $\Omega$ as $A^2(\Omega)$.
Recall that the Bergman space $A^2(\Omega)$ is the space of holomorphic functions on $\Omega$ that are square integrable on $\Omega$ under the standard Lebesgue measure.  The Bergman space is a closed subspace of $L^2(\Omega)$. Therefore there exists an orthogonal projection $P:L^2(\Omega)\rightarrow A^2(\Omega)$ called the Bergman projection.  The Hankel operator with symbol $\phi$ is defined as $H_{\phi}g=(I-P)(\phi g)$ for all $g\in A^2(\Omega)$.  If $\phi\in L^{\infty}(\Omega)$, then $H_{\phi}$ is a bounded operator, however, the converse is not necessarily true.  In one complex variable on the unit disk, Axler in \cite{Axl86} showed that the Hankel operator with symbol $\phi$ is bounded if and only if $\phi$ is in the Bloch space.  There are unbounded, holomorphic functions in the Bloch space, as it only specifies a growth rate of the derivative of the function near the boundary of the disk.  Namely, an analytic function $\phi$ is in the Bloch space if \[\sup\{(1-|z|^2)|\phi'(z)|:z\in \mathbb{D}\}<\infty.\]   
\par
Let $h\in A^2(\Omega)$ so that the Hankel operator $H_{\overline{h}}$ is compact on $A^2(\Omega)$.  The Hankel operator with an $L^2(\Omega)$ symbol may only be densely defined, since the product of $L^2$ functions may not be in $L^2$.  However, if compactness of the Hankel operator is also assumed, then the Hankel operator with an $L^2$ symbol is defined on all of $A^2(\Omega)$.  

\par
We wish to use the geometry of the boundary of $\Omega$ to give conditions on $h$.  For example, if $\Omega$ is the bidisk, Le in \cite[Corollary 1]{Le} shows that if $h\in A^2(\mathbb{D}^2)$ such that $H_{\overline{h}}$ is compact on $A^2(\mathbb{D}^2)$ then $h\equiv c$ for some $c\in \mathbb{C}$.  In one variable, Axler in \cite{Axl86} showed that $H_{\overline{g}}$ is compact on $A^2(\mathbb{D})$ if and only if $g$ is in the little Bloch space.  That is, $\lim_{|z|\rightarrow 1^-}(1-|z|^2)|g'(z)|=0$.  If the symbol $h$ is smooth up to the boundary of a smooth bounded convex domain in $\mathbb{C}^2$, {\v{C}}u{\v{c}}kovi{\'c} and
{\c{S}}ahuto{\u{g}}lu in \cite{cs} showed that Hankel operator $H_h$ is compact if and only if $h$ is holomorphic along analytic disks in the boundary of the domain.

\par  
In this paper we will use the following notation.  \[S_{t}=\{z\in \mathbb{C}:|z|=t\},\]
\[\mathbb{T}^2=S_1\times S_1=\{z\in \mathbb{C}:|z|=1\}\times \{w\in \mathbb{C}:|w|=1\},\]
\[\mathbb{D}_r=\{z\in \mathbb{C}:|z|<r\}\] for any $r,t>0$. If $r=1$ we write
\[\mathbb{D}=\{z\in \mathbb{C}:|z|<1\}.\]
We say $\Delta\subset b\D$ is an analytic disk if there exists a function $h=(h_1,h_2):\mathbb{D}\rightarrow b\Omega$ so that each component function is holomorphic on $\mathbb{D}$ and the image $h(\mathbb{D})=\Delta$.  An analytic disk is said to be trivial if it is degenerate (that is, $\Delta=(c_1,c_2)$ for some constants $c_1$ and $c_2$).

\par
In \cite{ClosSahut} we considered bounded convex Reinhardt domains in $\mathbb{C}^2$. 
We characterized non trivial analytic disks in the boundary of such domains. 

We defined \[\Gamma_{\Omega}=\overline{\bigcup \{\phi(\mathbb{D})|\phi:\mathbb{D}\rightarrow b\Omega \, \text{are holomorphic, non-trivial}\}}\] 
and showed that \[\Gamma_{\Omega}=\Gamma_1\cup \Gamma_2\] where
either $\Gamma_1=\emptyset$ or
\[\Gamma_1=\overline{\mathbb{D}_{r_1}}\times S_{s_1}\]
and likewise either $\Gamma_2=\emptyset$ or 
\[\Gamma_2=S_{s_2}\times \overline{\mathbb{D}_{r_2}}\]
for some $r_1,r_2,s_1,s_2>0$.

\begin{remark}
	We only consider domains in $\mathbb{C}^2$ as opposed to domains in $\mathbb{C}^n$ for $n\geq 3$ because a full geometric characterization of analytic structure in higher dimensions is unknown.
\end{remark}

The main results are the following theorems.

\begin{theorem}\label{main}
	Let $\Omega\subset \mathbb{C}^2$ be a bounded convex Reinhardt domain.  Let $f\in A^2(\Omega)$ so that $H_{\overline{f}}$ is compact on $A^2(\Omega)$.  If $\Gamma_1\neq \emptyset$, then $f$ is a function of $z_2$ alone.  If $\Gamma_2\neq\emptyset$, then $f$ is a function of $z_1$ alone. 
\end{theorem}

\begin{corollary}\label{one}
	Let $\Omega\subset \mathbb{C}^2$ be a bounded convex Reinhardt domain.  Suppose $\Gamma_1\neq \emptyset$ and $\Gamma_2\neq \emptyset$.  Let $f\in A^2(\Omega)$ so that $H_{\overline{f}}$ is compact on $A^2(\Omega)$.  Then there exists $c\in \mathbb{C}$ so that $f\equiv c$.
\end{corollary}

\begin{theorem}\label{main2}
	Let $\Omega\subset \mathbb{C}^2$ be a $C^{\infty}$-smooth bounded pseudoconvex complete Reinhardt domain.  Let $f\in A^2(\Omega)$ such that $H_{\overline{f}}$ is compact on $A^2(\Omega)$.  If either of the following conditions hold:
	\begin{enumerate}
		\item There exists a holomorphic function $F=(F_1,F_2):\mathbb{D}\rightarrow b\Omega$ so that both $F_1$ and $F_2$ are not identically constant.
		\item $\Gamma_1\neq \emptyset$ and $\Gamma_2\neq \emptyset$
	\end{enumerate}
Then $f\equiv c$ for some $c\in \mathbb{C}$.	
\end{theorem}

\section{Preliminary Lemmas}
As a bit of notation to simplify the reading, we will use the multi-index notation.
That is, we will write\[z=(z_1,z_2)\] and \[z^{\alpha}=z_1^{\alpha_1}z_2^{\alpha_2}\] and $|\alpha|=\alpha_1+\alpha_2$.  We say $\alpha=\beta$ if $\alpha_1=\beta_1$ and $\alpha_2=\beta_2$.  If either $\alpha_1\neq \beta_1$ or $\alpha_2\neq \beta_2$ we say $\alpha\neq \beta$.

It is well known that for bounded complete Reinhardt domains in $\mathbb{C}^2$, the monomials \[\left\{\frac{z^{\alpha}}{\|z^{\alpha}\|_{L^2(\Omega)}}:\alpha\in \mathbb{Z}^2_+\right\}\] form an orthonormal basis for $A^2(\Omega)$.

We denote \[\frac{z^{\alpha}}{\|z^{\alpha}\|_{L^2(\Omega)}}=e_{\alpha}(z)\]

\begin{definition}
 For $\beta=(\beta_1,\beta_2)\in \mathbb{Z}^2$, we define \[G_{\beta}:=\left\{\psi\in L^2(\Omega):\psi(\zeta z)=\zeta^{\beta}\psi(z)\,\text{a.e.}\,z\in \D\,\text{a.e}\, \zeta \in \mathbb{T}^2\right\}\]

\end{definition}

Note this definition makes sense in the case $\Omega$ is a Reinhardt domain, and is the same as the definition of quasi-homogeneous functions in \cite{Le}.

\begin{lemma}\label{lem1}
Let $\Omega\subset \mathbb{C}^2$ be a bounded complete Reinhardt domain.  $G_{\alpha}$ as defined above are closed subspaces of $L^2(\Omega)$ and for $\alpha\neq \beta$,
\[G_{\alpha}\perp G_{\beta}\]
\end{lemma}

\begin{proof}
 The proof that $G_{\beta}$ is a closed subspace of $L^2(\Omega)$ is similar to \cite{Le}.  Without loss of generality, suppose 
 $\alpha_1\neq \beta_1$.  Since $\Omega$ is a complete Reinhardt
 domain, one can 'slice' the domain similarly to \cite{CuckovicSahutoglu14}.  That is, 
 \[\Omega=\bigcup_{z_2\in H_{\Omega}}(\Delta_{|z_2|}\times \{z_2\})\] where $H_{\Omega}\subset \mathbb{C}$ is a disk centered at
 $0$ and \[\Delta_{|z_2|}=\{z\in \mathbb{C}:|z|<r_{|z_2|}\}\] is a disk with radius depending on $|z_2|$.  As we shall see, the proof relies on the radial symmetry of both $H_{\Omega}$ and $\Delta_{|z_2|}$.
 
 Let $f\in G_{\alpha}$, $g\in G_{\beta}$, $z_1=r_1\zeta_1$, $z_2=r_2\zeta_2$ for $(\zeta_1,\zeta_2)\in \mathbb{T}^2$, and $r_1,r_2\geq 0$.  Then we have 
 \begin{align*}
  &\langle f,g\rangle\\
  =&\int_{\Omega}f(z)\overline{g(z)}dV(z)\\
  =&\int_{H_{\Omega}}\int_{0\leq r_1\leq r_{|z_2|}}\int_{\mathbb{T}}\zeta_1^{\alpha_1}\overline{\zeta_1}^{\beta_1}f(r_1,z_2)\overline{g(r_1,z_2)}r_1d\sigma(\zeta_1)dr_1dV(z_2).\\
 \end{align*}
Since $\alpha_1\neq \beta_1$, \[\int_{\mathbb{T}}\zeta_1^{\alpha_1}\overline{\zeta_1}^{\beta_1}d\sigma(\zeta_1)=0.\]
This completes the proof.
 \end{proof}

In the case of a bounded convex Reinhardt domain in $\mathbb{C}^2$, one can use the 'slicing' approach in \cite{CuckovicSahutoglu14} to expilictly compute $P(\overline{z}^j e_{n})$.

\begin{lemma}\label{express}
	Let $\Omega\subset \mathbb{C}^2$ be a bounded complete Reinhardt domain.  Then the Hankel operator with symbol $\overline{z}^j\overline{w}^k$ applied to the orthonormal basis vector $e_{n}$ has the following form:
	\[H_{\overline{z}^j}e_{n}(z)=\frac{\overline{z}^{j}z^{n}}{\|z^{n}\|}\]
	if either $n_1-j_1<0$ or $n_2-j_2<0$.  If $n_1-j_1\geq 0$ and $n_2-j_2\geq 0$ then we can express the Hankel operator applied to the standard orthonormal basis as
	\[H_{\overline{z}^j}e_{n}(z)=\frac{\overline{z}^{j}z^{n}}{\|z^{n}\|}-\frac{z^{n-j}\|z^{n}\|}{\|z^{n-j}\|^2}.\]
	Furthermore, for any monomial \[\overline{w}^jw^n\in G_{n-j}\] the projection \[(I-P)(\overline{w}^jw^n)\in G_{n-j}.\]
\end{lemma}

\begin{proof}
	We have 
	\begin{align*}
	&P(\overline{z}^j e_{n})(z)=\\
	&=\int_{\D}\overline{w}^j \frac{w^{n}}{\|w^{n}\|}\sum_{l\in \mathbb{Z}^2_+}\overline{e_{l}(w)}e_{l}(z)dV(z,w)\\
	&=\int_{H_{\Omega}}\int_{w_1\in \Delta_{|w_2|}}\overline{w_1}^{j_1}\overline{w_2}^{j_2}\frac{w_1^{n_1}w_2^{n_2}}{\|z^n\|}\sum_{l_1,l_2=0}^{\infty}\overline{e_{l_1,l_2}(w_1,w_2)}e_{l_1,l_2}(z_1,z_2)dA_1(w_1)dA_2(w_2)\\
	&=\sum_{l_1,l_2=0}^{\infty}\frac{z_1^{l_1}z_2^{l_2}}{\|z^{n}\|\|\|z^{l}\|^2}\int_{H_{\Omega}}\overline{w_2}^{j_2+l_2}w_2^{n_2}\int_{w_1\in \Delta_{|w_2|}}\overline{w_1}^{j_1+l_1}w_1^{n_1}dA_1(w_1)dA_2(w_2).
	\end{align*}
	Converting to polar coordinates and using the orthogonality of $\{e^{in\theta}:n\in \mathbb{Z}\}$ and the fact that \[\int_{w_1\in \Delta_{|w_2|}}\overline{w_1}^{j_1+l_1}w_1^{n_1}dA_1(w_1)\] is a radial function of $w_2$ and $H_{\Omega}$ is radially symmetric, we have the only non-zero term in the previous sum is when $n_2-j_2=l_2$ and $n_1-j_1=l_1$.  Therefore, we have 
	$P(\overline{w}^je_{n})(z)=0$ if $n_2-j_2<0$ or $n_1-j_1<0$.  Otherwise,
	if $n_2-j_2\geq 0$ and $n_1-j_1\geq 0$, we have 
	\[P(\overline{w}^j e_{n})(z)=\frac{z^{n-j}\|z^{n}\|}{\|z^{n-j}\|^2}.\]  Therefore, we have
	\[H_{\overline{w}^j}e_{n}(z)=\frac{\overline{z}^{j}z^{n}}{\|z^{n}\|}-\frac{z^{n-j}\|z^{n}\|}{\|z^{n-j}\|^2}\]
	if $n_2-k\geq 0$ and $n_1-j\geq 0$ otherwise 
	\[H_{\overline{w}^j}e_{n}(z)=\frac{\overline{z}^{j}z^{n}}{\|z^{n}\|}\]
	if either $n_2-k<0$ or $n_1-j<0$.
	This also shows that the subspaces $G_{\alpha}$ remain invariant under the projection $(I-P)$, at least for monomial symbols.
\end{proof}

 \begin{lemma}\label{lem2}
  For every $\alpha\geq 0$, the product Hankel operator \[H^*_{\overline{z}^{\alpha}} H_{\overline{z}^{\alpha}}:A^2(\Omega)\rightarrow A^2(\Omega)\] is a diagonal operator with 
  respect to the standard orthonormal basis \[\{e_{j}: j\in \mathbb{Z}^2_+\}.\]

 \end{lemma}

\begin{proof}

Assume without loss of generality, $j\neq l$. We have 
\begin{align*}
&\langle  H_{\overline{z}^{\alpha}}^* H_{\overline{z}^{\alpha}}e_{j},e_{l}\rangle\\
=&\langle  H_{\overline{z}^{\alpha}}e_{j},  H_{\overline{z}^{\alpha}}e_{l}\rangle\\
=&\langle (I-P)(\overline{z}^{\alpha}e_{j}),\overline{z}^{\alpha}e_{l}\rangle.\\
\end{align*}

We have $\overline{z}^{\alpha}e_j\in G_{j-\alpha}$, $\overline{z}^{\alpha}e_l\in G_{l-\alpha}$.  By Lemma \ref{express}, \[(I-P)\overline{z}^{\alpha}e_j\in G_{j-\alpha}.\]  By Lemma \ref{lem1},
$G_{\alpha}$ are mutually orthogonal.  Therefore, \[\langle (I-P)(\overline{z}^{\alpha}e_{j}),\overline{z}^{\alpha}e_{l}\rangle=0\] unless $j=l$.

\end{proof}

Using Lemma \ref{express} and Lemma \ref{lem2}, let us compute the eigenvalues of 
\[H^*_{\overline{z}^{\alpha}}H_{\overline{z}^{\alpha}}.\]   
Let us first assume $n-\alpha\geq 0$.  We have 
\begin{align*}
\langle H^*_{\overline{z}^{\alpha}}H_{\overline{z}^{\alpha}}e_{n},e_{n}\rangle
&=\langle \frac{\overline{z}^{\alpha}z^n}{\|z^{n}\|}-\frac{z^{n-\alpha}\|z^n\|}{\|z^{n-\alpha}\|^2},\frac{\overline{z}^{\alpha}z^{n}}{\|z^n\|}\rangle\\
&=\frac{\|\overline{z}^{\alpha}z^n\|^2}{\|z^n\|^2}
-\frac{\|z^n\|^2}{\|z^{n-\alpha}\|^2}.\\
\end{align*}
If $n-\alpha<0$, we have
\[\langle H^*_{\overline{z}^{\alpha}}H_{\overline{z}^{\alpha}}e_{n},e_{n}\rangle=\frac{\|\overline{z}^{\alpha}z^n\|^2}{\|z^n\|^2}.\]

\section{Proof of Theorem \ref{main}}

\begin{proof}
Assume $f\in A^2(\D)$ and $H_{\overline{f}}$ is compact on $A^2(\D)$.  Then, we can represent \[f=\sum_{j,k=0}^{\infty}c_{j,k,f}z_1^jz_2^k\] almost everywhere (with respect to the Lebesgue volume measure on $\Omega$).  Let \[\{e_m:m\in \mathbb{Z}^2_+\}\] be the standard orthonormal basis for $A^2(\Omega)$.  Then \[\|H_{\overline{f}}e_m\|^2\rightarrow 0\] as $|m|\rightarrow \infty$.  Using the mutual orthogonality of the subspaces $G_{\alpha}$, we get 
\begin{align*}
&\|H_{\overline{f}}e_m\|^2=\langle (I-P)(\overline{f}e_m),\overline{f}e_m\rangle\\
&=\langle \sum_{j,k=0}^{\infty}(I-P)(\overline{c_{j,k,f}}\overline{z_1}^j\overline{z_2}^ke_m), \sum_{s,p=0}^{\infty}\overline{c_{s,p,f}}\overline{z_1}^s\overline{z_2}^p e_m\rangle\\
&=\sum_{j,k=0}^{\infty}\|H_{\overline{c_{j,k,f}z_1^jz_2^k}}e_m\|^2\geq \|H_{\overline{c_{j,k,f}z_1^jz_2^k}}e_m\|^2 \\
\end{align*}
 for every $(j,k)\in \mathbb{Z}^2_+$.  Taking limits as $|m|\rightarrow \infty$, we have
 \[\lim_{|m|\rightarrow \infty}\|H_{\overline{c_{j,k,f}z_1^jz_2^k}}e_m\|^2=0\] for all $(j,k)\in \mathbb{Z}^2_+$.  The Hankel operators \[H^*_{\overline{c_{j,k,f}z_1^jz_2^k}}H_{\overline{c_{j,k,f}z_1^jz_2^k}}\] are diagonal by Lemma \ref{lem2}, with eigenvalues \[\lambda_{j,k,m}=\|H_{\overline{c_{j,k,f}z_1^jz_2^k}}e_m\|^2.\] This shows that \[H^*_{\overline{c_{j,k,f}z_1^jz_2^k}}H_{\overline{c_{j,k,f}z_1^jz_2^k}}\] are compact for every $(j,k)\in \mathbb{Z}^2_+$.  Then $H_{\overline{c_{j,k,f}z_1^jz_2^k}}$ are compact on $A^2(\D)$.

  Without loss of generality, assume $\Gamma_1\neq \emptyset$.  Then there exists a holomorphic function $F=(F_1,F_2):\mathbb{D}\rightarrow b\Omega$ so that $F_2$ is identically constant and $F_1$ is non-constant.  Therefore, by \cite{ClosSahut}, the composition \[\overline{c_{j,k,f}F_1(z)^jF_2(z)^k}\]must be holomorphic in $z$.  This cannot occur unless $\overline{c_{j,k,f}}=0$ for $j>0$.  Therefore, using the representation
  \[f=\sum_{j,k=0}^{\infty}c_{j,k,f}z_1^jz_2^k\] we have $f=\sum_{k=0}^{\infty}c_{0,k,f}z_2^k$ almost everywhere.  
  By holomorphicity of $f$ and the identity principle, this implies \[f\equiv \sum_{k=0}^{\infty}c_{0,k,f}z_2^k.\]  Hence $f$ is a function of only $z_2$.  The proof is similar if $\Gamma_2\neq \emptyset$.
\end{proof}

\subsection{Proof of Theorem \ref{main2}}
Using the same argument in the proof of Theorem \ref{main}, one can show compactness of $H_{\overline{f}}$ implies compactness of $H_{\overline{c_{j,k,f}z_1^jz_2^k}}$ for every $j,k\in \mathbb{Z}_+$.  Hence by \cite[Corollary 1]{cs}, for any holomorphic function $\phi=(\phi_1,\phi_2):\mathbb{D}\rightarrow b\D$, we have 
\[\overline{c_{j,k,f}}\overline{\phi_1}^j\overline{\phi_2}^k\] must be holomorphic.  If we assume condition two in Theorem \ref{main2}, then it follows that $f\equiv c_{0,0,f}$.
Assuming condition one in Theorem \ref{main2}, we may assume $\phi_1$ and $\phi_2$ are not identically constant.  Thus $c_{j,k,f}=0$ for $j>0$ or $k>0$ and so $f\equiv c_{0,0,f}$.

\section*{\ackname}
 I wish to thank S{\"o}nmez {\c{S}}ahuto{\u{g}}lu, Trieu Le, and the referee for their valuable comments on a preliminary version of this manuscript.

\end{document}